\documentclass[a4paper,12pt]{article}

\usepackage{amsmath,amsfonts,amssymb,amsthm}
\newtheorem{Theorem}{Theorem}[section]
\newtheorem*{Lemma}{Lemma}
\newtheorem*{Coro}{Corollary}
\numberwithin{equation}{section}
\newcommand{\ba}{\beta}
\newcommand{\ga}{\gamma}
\newcommand{\ka}{\kappa}
\newcommand{\la}{\lambda}
\newcommand{\ta}{\theta}
\newcommand{\om}{\omega}

\author{Michel Lassalle\\
\small Centre National de la Recherche Scientifique\\[-0.8ex]
\small Institut Gaspard-Monge, Universit\'e de Marne-la-Vall\'ee\\[-0.8ex]
\small 77454 Marne-la-Vall\'ee Cedex, France\\[-0.8ex]
\small \texttt{lassalle@univ-mlv.fr}\\[-0.8ex]
\small \texttt{http://igm.univ-mlv.fr/{\textasciitilde}lassalle}
\\[1.6ex]
Michael J. Schlosser\\
\small Fakult\"at f\"ur Mathematik, Universit\"at Wien\\[-0.8ex]
\small Nordbergstra{\ss}e 15, A-1090 Wien, Austria\\[-0.8ex]
\small \texttt{michael.schlosser@univie.ac.at}\\[-0.8ex]
\small \texttt{http://www.mat.univie.ac.at/{\textasciitilde}schlosse}}

\title{Recurrence formulas\\ for Macdonald polynomials of type $A$}

\begin{document}

\date{}
\maketitle

\begin{abstract}
We consider products of two Macdonald polynomials of type $A$,
indexed by dominant weights which are respectively a multiple
of the first fundamental weight and a weight having zero component
on the $k$-th fundamental weight. We give the explicit decomposition
of any Macdonald polynomial of type $A$ in terms of this basis.\\

\end{abstract}

\section{Introduction}

In the 1980's, I. G. Macdonald introduced a class of orthogonal
polynomials which are Laurent polynomials in several variables
and generalize the Weyl characters of compact simple
Lie groups~\cite{Ma3,Ma4,Ma5}. In the simplest situation,
given a root system $R$, these polynomials are elements of the
group algebra of the weight lattice of $R$, indexed by the
dominant weights, and depending on two parameters $(q,t)$.  

When $R$ is of type $A_n$, these Macdonald polynomials are
in bijective correspondence with the symmetric functions
$\mathcal{P}_{\la}(q,t)$ indexed by partitions, that were introduced
by Macdonald some years before~\cite{Ma1,Ma2}.
In fact, they correspond to $\mathcal{P}_{\la}(q,t)(x_1,\ldots,x_{n+1})$,
for a partition $\la=(\la_1,\ldots,\la_{n})$ of length $n$,
with the $n+1$ variables $(x_1,\ldots,x_{n+1})$ linked by the
condition $x_1\cdots x_{n+1}=1$.

The purpose of this article is to extend the result of~\cite{LS},
given for the symmetric functions $\mathcal{P}_{\la}(q,t)$, 
to the framework of the root system $A_n$. 

More precisely, in~\cite[Theorem~4.1]{LS} we obtained a recurrence formula giving the symmetric function $\mathcal{P}_{(\la_1,\ldots,\la_{n})}(q,t)$ as a sum
\begin{equation}
\mathcal{P}_{(\la_1,\ldots,\la_{n})}= \sum_{\ta\in\mathbb{N}^{n-1}}
C_{\ta_1,\ldots,\ta_{n-1}}
\mathcal{P}_{(\la_1+\ta_1,\ldots,\la_{n-1}+\ta_{n-1})}\,\mathcal{P}_{\la_{n}-|\ta|},
\end{equation}
with $|\ta| = \sum_{i = 1}^{n-1} \ta_i$ and $\mathbb{N}$ the set of non-negative integers.
This formula was obtained by inverting the ``Pieri formula'',
which conversely expresses the product
$\mathcal{P}_{(\la_1,\ldots,\la_{n-1})} \, \mathcal{P}_{\la_{n}}$ as a sum 
\begin{equation*}
\mathcal{P}_{(\la_1,\ldots,\la_{n-1})} \, \mathcal{P}_{\la_{n}}= \sum_{\ta\in \mathbb{N}^{n-1}}
c_{\ta_1,\ldots,\ta_{n-1}}\,
\mathcal{P}_{(\la_1+\ta_1,\ldots,\la_{n-1}+\ta_{n-1},\la_{n}-|\ta|)}.
\end{equation*}
Both expansions are identities between symmetric functions,
valid for any number of variables. 

These identities may also be written in terms of Macdonald polynomials
of type $A_n$. For this purpose let $\{\om_i, 1\le i\le n\}$ be the
$n$ fundamental weights of the root system $A_n$.
Let $P_\la$ denote the Macdonald polynomial associated with the
dominant weight $\la=\sum_{i=1}^n \la_i\om_i$.
The recurrence formula (1.1), written for $n+1$ variables
$(x_1,\ldots,x_{n+1})$ linked by $x_1\cdots x_{n+1}=1$, yields
\begin{equation}
P_{\la}= \sum_{\ta\in\mathbb{N}^{n-1}}
C_{\ta_1,\ldots,\ta_{n-1}}
P_{(\la_{n}-|\ta|)\om_1}\,P_{\mu},
\end{equation}
with
$\mu=\sum_{i=1}^{n-2} (\la_i+\ta_i-\ta_{i+1})\om_i+
(\la_{n-1}+\la_{n}+\ta_{n-1})\om_{n-1}$.
This alternative formulation is obvious and does not bring anything new. 

However the method of~\cite{LS}, when applied in the $A_n$ root system
framework, allows to get a much stronger result.
Indeed, let $k$ be a fixed integer with $1\le k\le n$.
In this paper we shall write the Macdonald polynomial $P_\la$
in terms of products $P_{r\om_1}P_\mu$, with
$\mu=\sum_{i=1}^{n} \mu_i\om_i$ and $\mu_k=0$.
There are $n$ such recurrence formulas, (1.2) being the
particular case $k=n$ of the latter.

This paper is organized as follows.
In Section~\ref{secmacd} we introduce our notation for the
root system $A_n$ and recall general facts about the
corresponding Macdonald polynomials.
Their Pieri formula, which involves a specific infinite
multidimensional matrix, is studied in Section~\ref{secpieri},
starting from the one given by Macdonald for the symmetric functions
$\mathcal{P}_\la(q,t)$~\cite[p.~340]{Ma2}.
In Section~\ref{secrec} we invert the Pieri matrix by applying a
particular multidimensional matrix inverse, given separately
in the Appendix. This matrix inverse is equivalent to one
previously obtained in~\cite[Section~2]{LS} by using operator methods.
As result of inverting the Pieri formula we obtain recurrence formulas
for $A_n$ Macdonald polynomials.
Finally, in Section~\ref{secex} we detail the examples of the
$A_2$ and $A_3$ cases and compare them to earlier results.

\smallskip
{\it Acknowledgemnts.}
We thank the anonymous referees for helpful comments.
The second author was partly supported by FWF Austrian Science Fund
grants P17563-N13 and S9607.
\normalsize

\newpage
\section{Macdonald polynomials of type $A$}
\label{secmacd}

The standard references for Macdonald polynomials associated with
root systems are~\cite{Ma3,Ma4,Ma5}. 

Let us consider the space $\mathbb{R}^{n+1}$ endowed with the usual
scalar product and the quotient space
$V=\mathbb{R}^{n+1}/\mathbb{R}(1,\ldots,1)$,
where $\mathbb{R}(1,\ldots,1)$ is the subspace spanned by the vector
$(1,\ldots,1)$. Let $\varepsilon_1,\ldots,\varepsilon_{n+1}$
denote the images in $V$ of the coordinate vectors of $\mathbb{R}^{n+1}$,
linked by $\sum_{i=1}^{n+1} \varepsilon_i=0$. 

The root system of type $A_n$ is formed by the vectors
$\{\varepsilon_i-\varepsilon_j, i\neq j\}$. The positive roots are
$\{\varepsilon_i-\varepsilon_j, i<j\}$ and the simple roots are
$\varepsilon_i-\varepsilon_{i+1}$ for $1\le i \le n$.
The Weyl group is the symmetric group $W=S_{n+1}$ acting by
permutation of the coordinates. 

The weight lattice $P$ is formed by integral linear combinations
of the fundamental weights $\{\om_i, 1\le i \le n\}$, defined by
$\omega_i=\varepsilon_1+\ldots+ \varepsilon_{i}$. Let $\omega_{i}=0$ for $i=0,n+1$. We denote by $P^+$
the set of dominant weights $\la=\sum_{i=1}^{n} \la_i\om_i$, which
are non-negative integral linear combinations of the fundamental weights. 

There is the following correspondence between dominant weights and
partitions. Given a dominant weight, if we write it as
\[\la=\sum_{i=1}^{n} \la_i\om_i=\sum_{i=1}^{n+1} \mu_i\varepsilon_i,\]
the sequence $\mu=(\mu_1,\ldots,\mu_{n+1})$ is a partition with length
$\le n+1$. We have 
\[\la_i=\mu_{i}-\mu_{i+1}\quad\textrm{and} \quad
\mu_i=\mu_{n+1}+\sum_{j=i}^{n}\la_j.\] 
Thus $\mu$ is defined up to $\mu_{n+1}$ and two partitions
$\mu, \nu$ correspond to the same weight $\la$ if and only if
$\mu_1-\nu_1=\cdots= \mu_{n+1}-\nu_{n+1}$.
We denote by $\mathcal{C}_\la$ the family of partitions thus defined.  

Let $A$ denote the group algebra over $\mathbb{R}$ of the free
Abelian group $P$. For each $\la \in P$ let $e^\la$ denote the
corresponding element of $A$, subject to the multiplication rule
$e^\la e^\mu = e^{\la +\mu}$. The set $\{e^\la, \la\in P\}$ forms an
$\mathbb{R}$-basis of $A$.

The Weyl group $W=S_{n+1}$ acts on $P$ and on $A$. Let $W\la$ denote the orbit of $\la \in P$ and
$A^{W}$ the subspace of $W$-invariants in $A$.
There are two important bases of $A^{W}$, both indexed by dominant weights. 
The first one is given by the orbit-sums
\[m_\la = \sum_{\mu\in W\la}e^{\mu}.\]
The second one is provided by the Weyl characters 
\[\chi_\la = \delta^{-1} \sum_{w\in W}\mathrm{det}(w)e^{w(\la + \rho)},\]
with $\rho=\sum_{i=1}^n (n-i+1)\varepsilon_i$ and
$\delta=\sum_{w\in W}\mathrm{det}(w)e^{w(\rho)}$.
The Macdonald polynomials $\{P_\la, \la \in P^+\}$ form another basis,
defined as the eigenvectors of a specific self-adjoint operator
(which we do not describe here).

{}For $1\le i \le n+1$  define $x_i=e^{\varepsilon_i}$, so that the
variables $x_i$ are linked by $x_1\cdots x_{n+1}=1$. Then $\delta$
is the Vandermonde determinant $\prod_{i < j} (x_i-x_j)$.
There is a correspondence between $A^{W}$ and the symmetric polynomials
restricted to $n+1$ variables $x=(x_1,\ldots,x_{n+1})$ linked by the
previous condition. 

In terms of bases this correspondence may be described as follows.
Let $\la$ be any dominant weight and $x_1\cdots x_{n+1}=1$. All monomial symmetric functions
$m_\mu(x_1,\ldots,x_{n+1})$ with $\mu \in \mathcal{C}_\la$ are equal
and their common value is the orbit-sum $m_\la$.
Similarly, the Weyl character $\chi_\la$ is the common value of the
Schur functions $s_\mu(x_1,\ldots,x_{n+1})$, $\mu \in \mathcal{C}_\la$,
whereas the Macdonald polynomial $P_\la$ is the common value of the
symmetric polynomials $\mathcal{P}_{\mu}(q,t)(x_1,\ldots,x_{n+1})$, with $\mu \in \mathcal{C}_\la$ and $\mathcal{P}_\mu(q,t)$ the symmetric function studied in Chapter 6 of~\cite{Ma2}. 

Given a positive integer $r$ and a dominant weight $\la$, the
``Pieri formula'' expands the product
\[P_{r\omega_1} \, P_{\la} =\sum_{\rho} c_\rho \, P_{\la+\rho},\]
in terms of Macdonald polynomials,
where the range of $\rho$ and the values of the coefficients
$c_\rho$ are to be determined. 

Let $Q$ denote the root lattice, spanned by the simple roots.
For any vector $\tau$, define
 \[\Sigma(\tau)= C(\tau) \cap (\tau+Q)\]
with $C(\tau)$ the convex hull of the Weyl group orbit of $\tau$.
Since the orbit of $\om_1=\varepsilon_1$ is the set
$\{\varepsilon_i=\om_i-\om_{i-1},1\le i \le n+1\}$, it is clear that $\Sigma(r\omega_1)$ is formed by vectors
\[\sum_{i=1}^{n+1} \ta_i(\om_i-\om_{i-1})=
\sum_{i=1}^n (\ta_i-\ta_{i+1})\om_i,\]
with $\ta=(\ta_1,\ldots,\ta_{n+1})\in \mathbb{N}^{n+1}$ and
$|\ta|=\sum_{i=1}^{n+1} \ta_i=r$. 

By general results~\cite[(5.3.8), p. 104]{Ma5}, it is known that
the sum on the right-hand side of the Pieri formula is restricted
to vectors $\rho$ such that $\rho \in \Sigma(r\omega_1)$ and
$\la+\rho\in P^+$. In the next section we shall give a direct proof
of this result and make the value of the coefficient $c_\rho$ explicit.

\section{Pieri formula}
\label{secpieri}

Let $0< q <1$.
For any integer $r$, the classical $q$-shifted factorial $(u;q)_r$
is defined by
\begin{equation*}
(u;q)_{\infty}=\prod_{j\ge 0}(1-uq^j),\qquad
(u;q)_r=(u;q)_{\infty}/(uq^r;q)_{\infty}.
\end{equation*}

Let $u=(u_1,\ldots,u_{m})$ be $m$ indeterminates and
$\ta=(\ta_1,\ldots,\ta_{m})\in \mathbb{N}^{m}$.
For clarity of display, throughout this paper,
any time such a pair $(u,\ta)$ is given, we shall implicitly
assume $m$ auxiliary variables $v=(v_1,\ldots,v_{m})$
to be defined by $v_i=q^{\ta_i}u_i$. 

Macdonald polynomials of type $A_n$ satisfy the following Pieri formula.
\begin{Theorem}\label{theopieri}
Let  $\la=\sum_{i=1}^{n} \la_i\om_{i}$ be a dominant weight
and $r\in\mathbb{N}$. For any $1 \le i \le n+1$ define
\[u_i=q^{\sum_{j=i}^{n}\la_j}  t^{-i},\]
and for $\ta \in \mathbb{N}^{n+1}$,
\[d_\ta (u_1,\ldots,u_{n+1};r)=\frac{(q;q)_{r}}{(t;q)_{r}}\,
\prod_{j=1}^{n+1}
\frac{(t;q)_{\ta_j}}{(q;q)_{\ta_j}}\,
\prod_{1\le i < j \le n+1}
\frac{(tv_i/v_j;q)_{\ta_j}}{(qv_i/v_j;q)_{\ta_j}}\,
\frac{(qu_i/tv_j;q)_{\ta_j}}{(u_i/v_j;q)_{\ta_j}}.\]
We have
\[ P_{r\om_1} \, P_{\la}
=\sum_{\begin{subarray}{c}\ta\in \mathbb{N}^{n+1}\\ |\ta|=r\end{subarray}}
d_\ta (u_1,\ldots,u_{n+1};r) \, P_{\la+\rho},\]
with $\rho=\sum_{i=1}^n (\ta_i-\ta_{i+1})\om_i$.
\end{Theorem}
\begin{proof}
In a first step, we write the Pieri formula for arbitrary
$\mathcal{P}_\mu (q,t)$ with $\mu=(\mu_1,\ldots,\mu_n)$ being a partition
having length $\le n$.
We start from~\cite[p.~340, Eq.~(6.24)(i)]{Ma2} and
\cite[p.~342, Example~2(a)]{Ma2}. Replacing $g_r$ by
$(t;q)_r/(q;q)_r\, \mathcal{P}_{(r)}$ we have
\[\mathcal{P}_{(r)}\, \mathcal{P}_{\mu}=\sum_{\kappa \supset\mu}
\varphi_{\kappa/\mu}\,\mathcal{P}_{\kappa},\]
where the skew-diagram $\kappa-\mu$ is a horizontal
$r$-strip, i.e.\ has at most one node in each column.
The Pieri coefficient $\varphi_{\kappa/\mu}$ is given by
\[\frac{(t;q)_{r}}{(q;q)_{r}}\,\varphi_{\kappa/\mu}=
\prod_{1\le i\le j\le l(\kappa)}
\frac{f(q^{\kappa_{i}-\kappa_{j}}t^{j-i})}
{f(q^{\kappa_{i}-\mu_{j}}t^{j-i})}\,
\frac{f(q^{\mu_{i}-\mu_{j+1}}t^{j-i})}
{f(q^{\mu_{i}-\kappa_{j+1}}t^{j-i})}
=\prod_{1\le i\le j\le l(\kappa)}
\frac{w_{\kappa_j-\mu_j}(q^{\kappa_{i}-\kappa_{j}}t^{j-i})}
{w_{\kappa_{j+1}-\mu_{j+1}}(q^{\mu_{i}-\kappa_{j+1}}t^{j-i})},\]
with $f(u)=(tu;q)_{\infty}/(qu;q)_{\infty}$ and
$w_{s}(u)=(tu;q)_{s}/(qu;q)_{s}$.

Since $\kappa-\mu$ is a horizontal strip, the length $l(\kappa)$ of
$\kappa$ is at most equal to $n+1$, so we can write
$\kappa =(\mu_1+\ta_1,\ldots,\mu_n+\ta_n,\ta_{n+1})$, with $|\ta|=r$.
Then
\begin{align*}
\frac{(t;q)_{r}}{(q;q)_{r}}\,\varphi_{\kappa/\mu}&=
\prod_{1\le i\le j\le l(\kappa)}w_{\ta_{j}}(q^{\kappa_{i}-\kappa_{j}}t^{j-i})
\prod_{1\le i<j\le l(\kappa)+1}\big(w_{\ta_{j}}
(q^{\mu_{i}-\kappa_{j}}t^{j-i-1})\big)^{-1}\\
&=\prod_{j=1}^{n+1}
\frac{(t;q)_{\ta_j}}{(q;q)_{\ta_j}}\,
\prod_{1\le i < j \le n+1}
\frac{(tv_i/v_j;q)_{\ta_j}}{(qv_i/v_j;q)_{\ta_j}}\,
\frac{(qu_i/tv_j;q)_{\ta_j}}{(u_i/v_j;q)_{\ta_j}},
\end{align*}
where for $1\le i \le n+1$ we set $u_i=q^{\mu_i} t^{-i}$
and $v_i=q^{\kappa_i}t^{-i}=q^{\ta_i} u_i$.

In a second step we translate this result in terms of $A_n$
Macdonald polynomials. Given the dominant weight $\la$,
we choose $\mu=(\mu_1,\ldots,\mu_{n+1})$ to be the unique element
of $\mathcal{C}_\la$ such that $\mu_{n+1}=0$, i.e.\ with length $\le n$.
For $1\le i\le n$ we have $\mu_i=\sum_{j=i}^{n}\la_j$.
As for the partition $\kappa$ (with length $\le n+1$),
it belongs to $\mathcal{C}_\sigma$ with
$\sigma=\sum_{k=1}^n (\kappa_{k}-\kappa_{k+1})\om_k=
\sum_{k=1}^n (\la_{k}+\ta_k-\ta_{k+1})\om_k$. Hence the statement.
\end{proof}

\noindent \textit{Remark.} On the right-hand side of the Pieri formula, the condition $\la+\rho\in P^+$ is necessarily satisfied as soon as $d_\ta (u_1,\ldots,u_{n+1};r)\neq 0$. Using the correspondence between dominant weights and partitions, this may be verified on the Pieri formula 
\[\mathcal{P}_{(r)}\, \mathcal{P}_{\mu}=\sum_{\kappa =(\mu_1+\ta_1,\ldots,\mu_n+\ta_n,\ta_{n+1})}\varphi_{\kappa/\mu}\,\mathcal{P}_{\ka}.\] 
We only have to show that $\varphi_{\kappa/\mu}$ necessarily vanishes when the multi-integer $\ka$ is not a partition. But then there is an index $i$ such that $\ka_i <\ka_{i+1}$ so that the factor $(qu_i/tv_{i+1};q)_{\ta_{i+1}}$ in $\varphi_{\kappa/\mu}$ writes out as
\[(1-q^{1+\mu_i-\ka_{i+1}})\cdots(1-q^{\mu_i-\mu_{i+1}}).\]
Due to $\ka_i <\ka_{i+1}$ this product would be $\neq 0$ only if $\mu_i< \mu_{i+1}$, which is impossible since $\mu$ is a partition.
\smallskip

{}From now on, we fix some integer $1\le k\le n$.
Substituting $r-|\ta|$ for $\ta_k$, the Pieri formula
may be written in the more explicit form
\[P_{r\om_1} \, P_{\la}=
\sum_{\begin{subarray}{c}\ta=(\ta_1,\ldots,\ta_{k-1},0,
\ta_{k+1},\ldots, \ta_{n+1})\in \mathbb{N}^{n} \\
|\ta|\le r\end{subarray}}
\hat{d_\ta} (u_1,\ldots,u_{n+1};r) \, P_{\la+\rho},\]
with 
\[\rho=\sum_{\begin{subarray}{c}1\le i \le n \vspace{.05 cm} \\
i\neq k-1,k\end{subarray}} (\ta_i-\ta_{i+1})\om_i+
\ta_{k-1}\,\om_{k-1}+(r-|\ta|)(\om_k-\om_{k-1})-\ta_{k+1}\,\om_{k},\]
and
\begin{multline*}
\hat{d_\ta} (u_1,\ldots,u_{n+1};r) =\frac{(q;q)_{r}}{(t;q)_{r}}\,
\frac{(t;q)_{r-|\ta|}}{(q;q)_{r-|\ta|}}\,
\prod_{\begin{subarray}{c}j=1\\
j \neq k\end{subarray}}^{n+1}
\frac{(t;q)_{\ta_j}}{(q;q)_{\ta_j}}\\\times
\prod_{\begin{subarray}{c}1\le i < j \le n+1 \\
j \neq k\end{subarray}}
\frac{(tv_i/v_j;q)_{\ta_j}}{(qv_i/v_j;q)_{\ta_j}}\,
\frac{(qu_i/tv_j;q)_{\ta_j}}{(u_i/v_j;q)_{\ta_j}}
\prod_{i=1}^{k-1}
\frac{(tv_i/v_k;q)_{r-|\ta|}}{(qv_i/v_k;q)_{r-|\ta|}}\,
\frac{(qu_i/tv_k;q)_{r-|\ta|}}{(u_i/v_k;q)_{r-|\ta|}}.
\end{multline*}
Here $u_i,v_i$ $(1\le i \le n+1)$ are as in Theorem~\ref{theopieri},
except $v_k=q^{r-|\ta|}u_k$. The sum is restricted to $|\ta|\le r$ since $1/(q;q)_s=0$ for $s<0$.

In a second step, we concentrate on the situation $\la_k=0$.
Then each term on the right-hand side vanishes unless $\ta_{k+1}=0$.
Indeed, if $\la_k=0$, one has $u_{k}=tu_{k+1}$ and
$v_{k+1}=q^{\ta_{k+1}}u_{k+1}$. Hence for $i=k$ and $j=k+1$
the factor $(qu_i/tv_j;q)_{\ta_j}$ evaluates as
\[(qu_k/tv_{k+1};q)_{\ta_{k+1}}=(q^{1-\ta_{k+1}};q)_{\ta_{k+1}}
=\delta_{\ta_{k+1},0}.\]
Therefore if $\la_{k}=0$ the Pieri formula can be written as
\[P_{r\om_1} \,P_{\la}=
\sum_{\begin{subarray}{c}\ta=(\ta_1,\ldots,\ta_{k-1},0,0,
\ta_{k+2},\ldots,\ta_{n+1})\in \mathbb{N}^{n-1} \\
|\ta|\le r\end{subarray}}
\tilde{d}_\ta (u_1,\ldots,u_{k-1},u_k,u_{k+2},\ldots,u_{n+1};k,r) \, \, P_{\la+\rho},\]
with
\[\rho=\sum_{\begin{subarray}{c}1\le i \le n \\
i\neq k-1,k,k+1\end{subarray}} (\ta_i-\ta_{i+1})\om_i+
\ta_{k-1}\,\om_{k-1}+(r-|\ta|)(\om_{k}-\om_{k-1})-\ta_{k+2}\,\om_{k+1},\]
and 
\begin{multline*}
\tilde{d}_\ta (u_1,\ldots,u_{k-1},u_k,u_{k+2},\ldots,u_{n+1};k,r)=\\
\frac{(q;q)_r}{(t;q)_r}\,
\frac{(t;q)_{r-|\ta|}}{(q;q)_{r-|\ta|}}\,
\prod_{\begin{subarray}{c}i=1 \vspace{.05 cm} \\
i\neq k,k+1\end{subarray}}^{n+1} \frac{(t;q)_{\ta_i}}{(q;q)_{\ta_i}}\,
\prod_{\begin{subarray}{c}1\le i < j \le n+1 \vspace{.05 cm} \\
i\neq k,k+1\\ j\neq k,k+1\end{subarray}}
\frac{(tv_i/v_j;q)_{\ta_j}}{(qv_i/v_j;q)_{\ta_j}}\,
\frac{(qu_i/tv_j;q)_{\ta_j}}{(u_i/v_j;q)_{\ta_j}}\\\times
\prod_{i=1}^{k-1}
\frac{(tv_i/v_k;q)_{r-|\ta|}}{(qv_i/v_k;q)_{r-|\ta|}}\,
\frac{(qu_i/tv_k;q)_{r-|\ta|}}{(u_i/v_k;q)_{r-|\ta|}}\,
\prod_{j=k+2}^{n+1}
\frac{(tv_k/v_j;q)_{\ta_j}}{(qv_k/v_j;q)_{\ta_j}}\,
\frac{(qu_k/t^2v_j;q)_{\ta_j}}{(u_k/tv_j;q)_{\ta_j}}.
\end{multline*}
Here the notations are the same as before, including $v_k=q^{r-|\ta|}u_k$. For $j\ge k+2$ we have used
\[
\frac{(tv_k/v_j;q)_{\ta_j}}{(qv_k/v_j;q)_{\ta_j}}\,
\frac{(qu_k/tv_j;q)_{\ta_j}}{(u_k/v_j;q)_{\ta_j}}\,
\frac{(tv_{k+1}/v_j;q)_{\ta_j}}{(qv_{k+1}/v_j;q)_{\ta_j}}\,
\frac{(qu_{k+1}/tv_j;q)_{\ta_j}}{(u_{k+1}/v_j;q)_{\ta_j}}=
\frac{(tv_k/v_j;q)_{\ta_j}}{(qv_k/v_j;q)_{\ta_j}}\,
\frac{(qu_k/t^2v_j;q)_{\ta_j}}{(u_k/tv_j;q)_{\ta_j}},\]
which is a direct consequence of $v_{k+1}=u_{k+1}=u_{k}/t$.

{}In a third step, we perform some relabelling in order to remove the two 0's appearing in $\ta$. For that purpose, for $n$ indeterminates $(u_0,u_1,\ldots,u_{n-1})$ and
$\ta=(\ta_1,\ldots,\ta_{n-1})\in \mathbb{N}^{n-1}$, we define
\begin{multline*}
D_{\ta} (u_0,u_1,\ldots,u_{n-1};k,r)=\\
(q/t)^{|\ta|}\,
\frac{(t^2u_0;q)_{|\ta|}}{(qtu_0;q)_{|\ta|}}\,
\prod_{i=1}^{n-1}
\frac{(t;q)_{\ta_i}}{(q;q)_{\ta_i}}\,
\frac{(q^{|\ta|+1}u_i;q)_{\ta_i}}
{(q^{|\ta|}tu_i;q)_{\ta_i}}\,
\prod_{1\le i<j\le n-1} \frac{(tv_i/v_j;q)_{\ta_j}}
{(qv_i/v_j;q)_{\ta_j}}\,
\frac{(qu_i/tv_j;q)_{\ta_j}}
{(u_i/v_j;q)_{\ta_j}}
\\\times \prod_{i=1}^{k-1}
\frac{(u_i/u_0;q)_{\ta_i}}{(qu_i/tu_0;q)_{\ta_i}}\,
\frac{(qu_i/tu_0;q)_{\ta_i-r+|\ta|}}{(u_i/u_0;q)_{\ta_i-r+|\ta|}}\,
\frac{(u_i/tu_0;q)_{\ta_i-r+|\ta|}}{(qu_i/t^2u_0;q)_{\ta_i-r+|\ta|}}\,
\prod_{i=k}^{n-1}\frac{(tu_i/u_0;q)_{\ta_i}}{(qu_i/u_0;q)_{\ta_i}}.
\end{multline*}
\begin{Lemma}
If we write
\begin{equation*}
w_i=\begin{cases}
q^{-r}t^{-2},&\qquad i=0,\\
q^{-r}u_i/tu_k,&\qquad 1\le i\le k-1,\\
q^{-r}u_{i+2}/tu_k,&\qquad k\le i\le n-1,
\end{cases}
\end{equation*}
we have
\begin{equation*}
D_{\ta} (w_0,w_1,\ldots,w_{n-1};k,r)=
\tilde{d}_{(\ta_1,\ldots,\ta_{k-1},0,0,\ta_k,\ldots,\ta_{n-1})}
(u_1,\ldots,u_{k-1},u_k,u_{k+2},\ldots,u_{n+1};k,r).
\end{equation*}
\end{Lemma}
\begin{proof}
Merely by substitution, and using $v_k=q^{r-|\ta|}u_k$, we only have to prove
\begin{multline*}
(q/t)^{|\ta|}\,
\frac{(q^{-r};q)_{|\ta|}}{(q^{1-r}/t;q)_{|\ta|}}\,
\prod_{j=k+2}^{n+1}
\frac{(q^{|\ta|-r+1}u_{j}/tu_k;q)_{\ta_j}}
{(q^{|\ta|-r}u_{j}/u_k;q)_{\ta_j}}\,\frac{(t^2u_{j}/u_k;q)_{\ta_j}}{(qtu_{j}/u_k;q)_{\ta_j}}\\\times
\prod_{i=1}^{k-1}
\frac{(q^{|\ta|-r+1}u_i/tu_k;q)_{\ta_i}}
{(q^{|\ta|-r}u_i/u_k;q)_{\ta_i}}\,
\frac{(tu_i/u_k;q)_{\ta_i}}{(qu_i/u_k;q)_{\ta_i}}\,
\frac{(qu_i/u_k;q)_{\ta_i-r+|\ta|}}{(tu_i/u_k;q)_{\ta_i-r+|\ta|}}\,
\frac{(u_i/u_k;q)_{\ta_i-r+|\ta|}}{(qu_i/tu_k;q)_{\ta_i-r+|\ta|}}=\\
\frac{(q;q)_r}{(t;q)_r}\,
\frac{(t;q)_{r-|\ta|}}{(q;q)_{r-|\ta|}}\,
\prod_{i=1}^{k-1}
\frac{(tv_i/q^{r-|\ta|}u_k;q)_{r-|\ta|}}{(qv_i/q^{r-|\ta|}u_k;q)_{r-|\ta|}}\,
\frac{(qu_i/tq^{r-|\ta|}u_k;q)_{r-|\ta|}}{(u_i/q^{r-|\ta|}u_k;q)_{r-|\ta|}}\\\times
\prod_{j=k+2}^{n+1}
\frac{(tq^{r-|\ta|}u_k/v_j;q)_{\ta_j}}{(q^{r-|\ta|+1}u_k/v_j;q)_{\ta_j}}\,
\frac{(qu_k/t^2v_j;q)_{\ta_j}}{(u_k/tv_j;q)_{\ta_j}}.
\end{multline*}
We have obviously
\[\frac{(q^{|\ta|-r+1}u_i/tu_k;q)_{\ta_i}}
{(q^{|\ta|-r}u_i/u_k;q)_{\ta_i}}\,
\frac{(u_i/u_k;q)_{\ta_i-r+|\ta|}}{(qu_i/tu_k;q)_{\ta_i-r+|\ta|}}=\frac{(qu_i/tq^{r-|\ta|}u_k;q)_{r-|\ta|}}{(u_i/q^{r-|\ta|}u_k;q)_{r-|\ta|}}.\]
Using the identities
\begin{align*}
\frac{(aq^{-n};q)_n}{(bq^{-n};q)_n}&=\frac{(q/a;q)_n}{(q/b;q)_n} \,(a/b)^n,\\
\frac{(a;q)_n}{(b;q)_n}\,\frac{(b;q)_{n-k}}{(a;q)_{n-k}}&=
\frac{(q^{1-n}/a;q)_k}{(q^{1-n}/b;q)_k} \,(a/b)^k,
\end{align*}
we get
\begin{align*}
\frac{(tu_i/u_k;q)_{\ta_i}}{(qu_i/u_k;q)_{\ta_i}}\,
\frac{(qu_i/u_k;q)_{\ta_i-r+|\ta|}}{(tu_i/u_k;q)_{\ta_i-r+|\ta|}}&=\frac{(q^{1-\ta_i}u_k/tu_i;q)_{r-|\ta|}}{(q^{-\ta_i}u_k/u_i;q)_{r-|\ta|}} \,(t/q)^{r-|\ta|}\\
&=\frac{(tv_i/q^{r-|\ta|}u_k;q)_{r-|\ta|}}{(qv_i/q^{r-|\ta|}u_k;q)_{r-|\ta|}}.
\end{align*}
Similarly we obtain
\begin{align*}
(t/q)^{\ta_j}\,\frac{(q^{|\ta|-r+1}u_{j}/tu_k;q)_{\ta_j}}
{(q^{|\ta|-r}u_{j}/u_k;q)_{\ta_j}}&=\frac{(tq^{r-|\ta|}u_k/v_j;q)_{\ta_j}}{(q^{r-|\ta|+1}u_k/v_j;q)_{\ta_j}}\\
(q/t)^{\ta_j}\,\frac{(t^2u_{j}/u_k;q)_{\ta_j}}{(qtu_{j}/u_k;q)_{\ta_j}}&=\frac{(qu_k/t^2v_j;q)_{\ta_j}}{(u_k/tv_j;q)_{\ta_j}}.\tag*{\hspace{-1em}\qedhere}
\end{align*}
\end{proof}

{}Finally we have proved the following Pieri formula.
\begin{Theorem}\label{theopierimac}
Let  $\la=\sum_{i=1}^{n} \la_i\om_{i}$ be a dominant weight
and $r\in\mathbb{N}$. Assume $\la_k=0$ for some fixed $1\le k\le n$.
Define
\begin{equation*}
u_i=\begin{cases}
q^{-r}t^{-2},&\qquad i=0,\\
q^{-r+\sum_{j=i}^{k-1}\la_j}t^{k-i-1},&\qquad 1\le i\le k-1,\\
q^{-r-\sum_{j=k+1}^{i+1}\la_j}t^{k-i-3},&\qquad k\le i\le n-1.
\end{cases}
\end{equation*}
We have
\[P_{r\om_1}\,P_{\la}=
\sum_{\begin{subarray}{c}\ta=(\ta_1,\ldots,\ta_{n-1})\in \mathbb{N}^{n-1} \\
|\ta|\le r\end{subarray}}
D_{\ta} (u_0,u_1,\ldots,u_{n-1};k,r)\, P_{\la+\rho},\]
with
\[\rho=\sum_{i=1}^{k-2} (\ta_i-\ta_{i+1})\om_i+
\ta_{k-1}\,\om_{k-1}+(r-|\ta|)(\om_{k}-\om_{k-1})
-\ta_{k}\,\om_{k+1}+
\sum_{i=k+2}^n (\ta_{i-2}-\ta_{i-1})\om_i.\]
\end{Theorem}
\noindent \textit{Remark.}
For $k=1, 2$ (resp. $k=n$, $n-1$) the first (resp. the last) sum in the above expression of $\rho$ must be understood as zero. This convention will be kept in the next sections.

\section{A recurrence formula}
\label{secrec}

Given two multi-integers $\ba=(\ba_1,\dots,\ba_{n-1})$,
$\ka=(\ka_1,\dots,\ka_{n-1})\in{\mathbb Z}^{n-1}$, we write 
$\ba\ge\ka$ for $\ba_i\ge \ka_i$ $(1\le i \le n-1)$.
We say that an infinite $(n-1)$-dimensional matrix 
$F=(f_{\ba\ka})_{\ba,\ka\in{\mathbb Z}^{n-1}}$ 
is lower-triangular if $f_{\ba\ka}=0$ unless 
$\ba\ge\ka$. When all $f_{\ka\ka}\ne 0$, 
there exists a unique lower-triangular matrix
$G=(g_{\ka\ga})_{\ka,\ga\in{\mathbb Z}^{n-1}}$ such that 
\begin{equation*}
\sum_{\ba\ge\ka\ge\ga}
f_{\ba\ka}\,
g_{\ka\ga}=\delta_{\ba\ga},
\end{equation*}
for all $\ba,\ga\in{\mathbb Z}^{n-1}$,
where $\delta_{\ba\ga}$ is the usual Kronecker symbol.
We refer to $F$ and $G$ as mutually inverse.

Such a pair of infinite multidimensional inverse matrices
is given in the Appendix, as a corollary of~\cite[Theorem~2.7]{LS}
(and, in fact, equivalent to the latter).
This result is essential for our purpose.

Given $n$ indeterminates $(u_0,u_1,\ldots,u_{n-1})$,
$\ta=(\ta_1,\ldots,\ta_{n-1})\in \mathbb{N}^{n-1}$, and
$k,r\in\mathbb{N}$ with $1\le k\le n$, we define
\begin{multline*}
C_{\ta_1,\ldots,\ta_{n-1}} (u_0,u_1,\ldots,u_{n-1};k,r)=\\
q^{|\ta|}\,
\frac{(t^2u_0;q)_{|\ta|}}{(qtu_0;q)_{|\ta|}}\,
\prod_{i=1}^{n-1}
\frac{(q/t;q)_{\ta_i}}{(q;q)_{\ta_i}}\,
\frac{(qu_i;q)_{\ta_i}}{(qtu_i;q)_{\ta_i}} \,
\prod_{1\le i<j\le n-1}\frac{(qv_i/tv_j;q)_{\ta_j}}
{(qv_i/v_j;q)_{\ta_j}}\,
\frac{(tu_i/v_j;q)_{\ta_j}}
{(u_i/v_j;q)_{\ta_j}}\\\times
\prod_{i=1}^{k-1}\frac{(u_i/tu_0;q)_{\ta_i}}{(qu_i/t^2u_0;q)_{\ta_i}}\,
\frac{(qtu_0/u_i;q)_r}{(t^2u_0/u_i;q)_r}\,
\frac{(tu_0/u_i;q)_r}{(qu_0/u_i;q)_r}\,
\prod_{i=k}^{n-1}\frac{(tu_i/u_0;q)_{\ta_i}}{(qu_i/u_0;q)_{\ta_i}}\,
\\\times
\frac{1}{\Delta(v)}\,\det_{1\le i,j\le n-1}\!\Bigg[v_i^{n-j-1}
\Bigg(1-t^{j-1}\frac{1-tv_i}{1-v_i}
\prod_{s=1}^{n-1}\frac{v_i-u_s}
{v_i-tu_s}\Bigg)\Bigg],
\end{multline*}
with $\Delta(v)$ the Vandermonde determinant
$\prod_{1\le i < j \le n-1} (v_i-v_j)$.
Here is our main result.
\begin{Theorem}
Let  $\la=\sum_{i=1}^{n} \la_i\om_{i}$ be a dominant weight.
Assume $\la_k=0$ for some fixed $1\le k\le n$.
For any positive integer $r\le \la_{k-1}$ the weight
\[\la^{(r)}=\la+r(\om_{k}-\om_{k-1})=\la+r\varepsilon_k\]
is dominant. Define
\begin{equation*}
u_i=\begin{cases}
q^{-r}t^{-2},&\qquad i=0,\\
q^{-r+\sum_{j=i}^{k-1}\la_j}t^{k-i-1},&\qquad 1\le i\le k-1,\\
q^{-r-\sum_{j=k+1}^{i+1}\la_j}t^{k-i-3},&\qquad k\le i\le n-1.
\end{cases}
\end{equation*}
We have
\[P_{\la^{(r)}}=
\sum_{\begin{subarray}{c}\ta=(\ta_1,\ldots,\ta_{n-1})\in \mathbb{N}^{n-1} \\
|\ta|\le r\end{subarray}}
C_{\ta} (u_0,u_1,\ldots,u_{n-1};k,r)\,P_{(r-|\ta|)\om_1}\,P_{\la+\rho},\]
with
\[\rho=\sum_{i=1}^{k-2} (\ta_i-\ta_{i+1})\om_i+
\ta_{k-1}\,\om_{k-1}-\ta_{k}\,\om_{k+1}+
\sum_{i=k+2}^n (\ta_{i-2}-\ta_{i-1})\om_i.\]
\end{Theorem}

\noindent \textit{Remark.}
The weight $\la+\rho$ has no component on $\om_k$. Further, similarly
as in Theorem~\ref{theopieri} (see the Remark following the proof of that
theorem), the condition $\la+\rho\in P^+$ is necessarily satisfied
in Theorem~\ref{theorec} as soon as
$C_{\ta} (u_0,u_1,\ldots,u_{n-1};k,r)\neq 0$.
We omit the details which involve a tedious case-by-case analysis.

\begin{proof}
We make use of the multidimensional matrix inverse
given in the Appendix.
Let $\ba=(\ba_1,\dots,\ba_{n-1})$,
$\ka=(\ka_1,\dots,\ka_{n-1})$,
$\ga=(\ga_1,\dots,\ga_{n-1})\in\mathbb Z^{n-1}$.
If we define
\begin{align*}
f_{\ba\ka}&=C_{\ba_1-\ka_1,\ldots,\ba_{n-1}-\ka_{n-1}}
\big(q^{|\ka|}u_0,q^{\ka_1+|\ka|}u_1,\ldots,
q^{\ka_{n-1}+|\ka|}u_{n-1};k,r-|\ka|),\\
g_{\ka\ga}&=D_{\ka_1-\ga_1,\ldots,\ka_{n-1}-\ga_{n-1}}
\big(q^{|\ga|}u_0,q^{\ga_1+|\ga|}u_1,\ldots,
q^{\ga_{n-1}+|\ga|}u_{n-1};k,r-|\ga|\big),
\end{align*}
by this result, the infinite lower-triangular
multidimensional matrices $(f_{\ba\ka})_{\ba,\ka\in{\mathbb Z}^{n-1}}$
and $(g_{\ka\ga})_{\ka,\ga\in{\mathbb Z}^{n-1}}$ are
mutually inverse.

Now let us replace in Theorem~\ref{theopierimac}
$\la_i$ by $\la_i+\ga_i-\ga_{i+1}$
for $1\le i\le k-2$,
$\la_{k-1}$ by $\la_{k-1}+\ga_{k-1}$, 
$\la_{k+1}$ by $\la_{k+1}-\ga_k$, 
$\la_i$ by $\la_i+\ga_{i-2}-\ga_{i-1}$ for $k+2\le i\le n$,
$r$ by $r-|\gamma|$, respectively. Then $u_0$ is replaced by $q^{|\ga|}u_0$,
and $u_i$ by $q^{\gamma_i+|\gamma|}u_i$ for $1 \le i \le n-1$.
In explicit terms, we are considering the identity
\[P_{(r-|\ga|)\om_1}\,P_{\la+\tilde{\gamma}}=
\sum_{\begin{subarray}{c}\ta=(\ta_1,\ldots,\ta_{n-1})\in \mathbb{N}^{n-1} \\
|\ta|\le r\end{subarray}}
D_{\ta} (q^{|\ga|}u_0,q^{\gamma_1+|\gamma|}u_1,\ldots,q^{\gamma_{n-1}+|\gamma|}u_{n-1};k,r-|\gamma|)\, P_{\la+\tilde{\gamma}+\rho},\]
with
\begin{equation*}
u_i=\begin{cases}
q^{-r}t^{-2},&\qquad i=0,\\
q^{-r+\sum_{j=i}^{k-1}\la_j}t^{k-i-1},&\qquad 1\le i\le k-1,\\
q^{-r-\sum_{j=k+1}^{i+1}\la_j}t^{k-i-3},&\qquad k\le i\le n-1,
\end{cases}
\end{equation*}
and
\begin{align*}
\rho&=\sum_{i=1}^{k-2} (\ta_i-\ta_{i+1})\om_i+
\ta_{k-1}\,\om_{k-1}+(r-|\ta|)(\om_{k}-\om_{k-1})
-\ta_{k}\,\om_{k+1}+
\sum_{i=k+2}^n (\ta_{i-2}-\ta_{i-1})\om_i,\\
\tilde{\gamma}&=\sum_{i=1}^{k-2}
 (\ga_i-\ga_{i+1})\om_{i}+\ga_{k-1}\,\om_{k-1}-\ga_k\,\om_{k+1}+
\sum_{i=k+2}^n(\ga_{i-2}-\ga_{i-1})\om_i.
\end{align*}

After substituting the summation indices $\ta_i\mapsto\ka_i-\ga_i$
for $1\le i\le n-1$,
we obtain exactly
\begin{equation*}
\sum_{\kappa\in\mathbb Z^{n-1}}g_{\kappa\gamma}y_{\kappa}=
w_{\gamma}\qquad\qquad (\gamma \in\mathbb Z^{n-1}),
\end{equation*}
with
\[y_{\kappa}=P_{\la+\tilde{\kappa}},\quad \quad
w_{\gamma}=P_{(r-|\gamma|)\om_1}\,P_{\la+\tilde{\gamma}},
\]
and
\begin{equation*}
\tilde{\kappa}=\sum_{i=1}^{k-2}
 (\ka_i-\ka_{i+1})\om_{i}+\ka_{k-1}\,\om_{k-1}+
(r-|\ka|)(\om_{k}-\om_{k-1})-\ka_k\,\om_{k+1}+
\!\sum_{i=k+2}^n(\ka_{i-2}-\ka_{i-1})\om_i.
\end{equation*}

This immediately yields the inverse relation
\begin{equation*}
\sum_{\beta\in\mathbb Z^{n-1}}f_{\beta\kappa}w_{\beta}=y_{\kappa}\qquad\qquad
(\kappa \in\mathbb Z^{n-1}).
\end{equation*}
We conclude by setting $\kappa_i=0$ for all $1\le i \le n-1$.
\end{proof}

{}Finally, by the substitutions 
$r\rightarrow\la_k$ and $\la_{k-1}\rightarrow\la_{k-1}+\la_k$,
we obtain the following very remarkable expansion.

\begin{Theorem}\label{theorec}
Let  $\la=\sum_{i=1}^{n} \la_i\om_{i}$ be a dominant weight and
$k \in \mathbb{N}$ fixed with $1\le k\le n$. Define
\begin{equation*}
u_i=\begin{cases}
q^{-\la_k}t^{-2},&\qquad i=0,\\
q^{\sum_{j=i}^{k-1}\la_j}t^{k-i-1},&\qquad 1\le i\le k-1,\\
q^{-\sum_{j=k}^{i+1}\la_j}t^{k-i-3},&\qquad k\le i\le n-1,
\end{cases}
\end{equation*}
and $\mu=\la-\la_k\,(\om_{k}-\om_{k-1})=\la-\la_k\,\varepsilon_k$. We have
\[P_{\la}=
\sum_{\begin{subarray}{c}\ta=(\ta_1,\ldots,\ta_{n-1})\in \mathbb{N}^{n-1} \\
|\ta|\le \la_k\end{subarray}}
C_{\ta} (u_0,u_1,\ldots,u_{n-1};k,\la_k)\,P_{(\la_k-|\ta|)\om_1}\,
P_{\mu+\rho},\]
with
\[\rho=\sum_{i=1}^{k-2} (\ta_i-\ta_{i+1})\om_i
+\ta_{k-1}\,\om_{k-1}-\ta_{k}\,\om_{k+1}+
\sum_{i=k+2}^n (\ta_{i-2}-\ta_{i-1})\om_i.\] 
\end{Theorem}

\noindent \textit{Remark.} Observe that the weights $\mu$ and
$\mu+\rho$ have no component on $\om_k$. 

\smallskip
The $k=n$ special case is worth writing out explicitly.
\begin{Coro}
Let  $\la=\sum_{i=1}^{n} \la_i\om_{i}$ be a dominant weight.
Define $u_0=q^{-\la_n}t^{-2}$ and $u_i=q^{\sum_{l=i}^{n-1}\la_l}t^{n-i-1}$
$(1\le i\le n-1)$. We have
\begin{equation*}
P_{\la}=
\sum_{\begin{subarray}{c}\ta=(\ta_1,\ldots,\ta_{n-1})\in \mathbb{N}^{n-1} \\
|\ta|\le \la_n\end{subarray}}
C_{\ta} (u_0,u_1,\ldots,u_{n-1};n,\la_n)\,
P_{(\la_{n}-|\ta|)\om_1}\,P_{\mu},
\end{equation*}
with $\mu=\sum_{i=1}^{n-2} (\la_i+\ta_i-\ta_{i+1})\om_i+
(\la_{n-1}+\la_{n}+\ta_{n-1})\om_{n-1}$.
\end{Coro}
The reader may check that this is exactly Theorem~4.1 of \cite{LS}
(with $n\mapsto n-1$),
written for $x_1\cdots x_{n+1}=1$, up to the normalization
$Q_{\la}= b_{\la} \,P_{\la}$ with
\begin{equation*}
b_{\la}=\prod_{1\le i\le j\le n}
\frac{(q^{\sum_{l=i}^{j-1}\la_l}t^{j-i+1};q)_{\la_j}}
{(q^{1+\sum_{l=i}^{j-1}\la_l}t^{j-i};q)_{\la_j}}=
\prod_{1\le i\le j\le n}
\frac{(tu_i/u_j;q)_{\la_j}}
{(qu_i/u_j;q)_{\la_j}},
\end{equation*}
where we set $u_n=1/t$.

\section{Examples}
\label{secex}

In this section we write out the formulas in Theorem~\ref{theorec}
explicitly for $n=2,3$.

\subsection{The root system $A_2$} 

{}For $k=2$ we have $u_0=q^{-\la_2}/t^2$, $u_1=q^{\la_1}$, and
\begin{multline*}
C_\ta(u_0,u_1;2,r)=
q^{\ta}\,
\frac{(t^2u_0;q)_\ta}{(qtu_0;q)_\ta}\,
\frac{(q/t;q)_\ta}{(q;q)_\ta}\,
\frac{(qu_1;q)_\ta}{(qtu_1;q)_\ta} \,
\frac{(u_1/tu_0;q)_\ta}{(qu_1/t^2u_0;q)_\ta}\\\times
\frac{(qtu_0/u_1;q)_r}{(t^2u_0/u_1;q)_r}\,
\frac{(tu_0/u_1;q)_r}{(qu_0/u_1;q)_r}\,
\Bigg(1-\frac{1-tv_1}{1-v_1}\frac{v_1-u_1}
{v_1-tu_1}\Bigg).
\end{multline*}
After some simplifications, we obtain
\[P_{\la_1\om_1+\la_2\om_2}=\sum_{\ta\in \mathbb{N}}
C^{(2)}_{\ta}(\la)\,P_{(\la_2-\ta)\om_1}\,P_{(\la_1+\la_2+\ta)\om_1},\]
with
\begin{align*}
C_{\ta}^{(2)}(\la)&=C_\ta(u_0,u_1;2,\la_2)\\&=t^\ta\,
\frac{(q^{\la_2-\ta+1};q)_{\ta}}{(tq^{\la_2-\ta};q)_{\ta}}\,
\frac{(1/t;q)_\ta}{(q;q)_\ta}\,
\frac{(q^{\la_1+1};q)_{\ta}}{(tq^{\la_1+1};q)_{\ta}}\,
\frac{(tq^{\la_1};q)_{\la_2+\ta}}{(q^{\la_1+1};q)_{\la_2+\ta}}\,
\frac{(tq^{\la_1+1};q)_{\la_2}}{(t^2q^{\la_1};q)_{\la_2}}\,
\frac{1-q^{\la_1+2\ta}}{1-q^{\la_1+\ta}}.
\end{align*}

This result may be compared with the Jing--J\'{o}zefiak
classical result~\cite{JJ}, more precisely its restriction
to three variables $(x_1,x_2,x_3)$ subject to $x_1x_2x_3=1$.
Namely, given a partition $(\mu_1,\mu_2)$, the Macdonald
symmetric function $\mathcal{P}_{(\mu_1,\mu_2)}(q,t)$ is given by
\begin{equation*}
\mathcal{P}_{(\mu_1,\mu_2)}= \sum_{\ta\in\mathbb{N}}
\mathcal{C}_{\ta} (\mu) \, \mathcal{P}_{(\mu_{2}-\ta)} \, \mathcal{P}_{(\mu_1+\ta)},
\end{equation*}
with 
\begin{multline*}
\mathcal{C}_{\ta} (\mu)=
\frac{(tq^{\mu_1-\mu_2+1};q)_{\mu_2}}{(t^2q^{\mu_1-\mu_2};q)_{\mu_2}}\,
\frac{(q^{\mu_2-\ta+1};q)_{\ta}}{(tq^{\mu_2-\ta};q)_{\ta}}\,
\frac{(tq^{\mu_1-\mu_2};q)_{\mu_2+\ta}}
{(q^{\mu_1-\mu_2+1};q)_{\mu_2+\ta}}\\\times
t^{\ta} \,\frac{(1/t;q)_{\ta}}{(q;q)_{\ta}}\,
\frac{(q^{\mu_1-\mu_2+1};q)_{\ta}}{(tq^{\mu_1-\mu_2+1};q)_{\ta}}\,
\frac{1-q^{\mu_1-\mu_2+2\ta}}{1-q^{\mu_1-\mu_2+\ta}}.
\end{multline*}
Our formula is equivalent to the main result of \cite{JJ} by the correspondence
$\la_1=\mu_1-\mu_2$, $\la_2=\mu_2$ between dominant weights and partitions,
recalled in Section 2.

{}For $k=1$ we have $u_0=q^{-\la_1}/t^2$, $u_1=q^{-\la_1-\la_2}/t^3$, and
\[C_\ta(u_0,u_1;1,r)=
q^{\ta}\,
\frac{(t^2u_0;q)_{\ta}}{(qtu_0;q)_{\ta}}\,
\frac{(q/t;q)_{\ta}}{(q;q)_{\ta}}\,
\frac{(qu_1;q)_{\ta}}{(qtu_1;q)_{\ta}}\,
\frac{(tu_1/u_0;q)_{\ta}}{(qu_1/u_0;q)_{\ta}}\,
\Bigg(1-\frac{1-tv_1}{1-v_1}\frac{v_1-u_1}
{v_1-tu_1}\Bigg).\]
After some simplifications, we obtain
\[P_{\la_1\om_1+\la_2\om_2}=
\sum_{\ta\in \mathbb{N}}
C^{(1)}_{\ta}(\la)\,P_{(\la_1-\ta)\om_1}\,P_{(\la_2-\ta)\om_2},\]
with
\begin{align*}
C_{\ta}^{(1)}(\la)&=C_\ta(u_0,u_1;1,\la_1)\\&=
t^{\ta}\,\frac{(1/t;q)_\ta}{(q;q)_\ta}\,
\frac{(q^{\la_1};1/q)_\ta}{(tq^{\la_1-1};1/q)_\ta}\,
\frac{(q^{\la_2};1/q)_\ta}{(tq^{\la_2-1};1/q)_\ta}\,
\frac{(t^3q^{\la_1+\la_2-1};1/q)_\ta}{(t^2q^{\la_1+\la_2-1};1/q)_\ta}\,
\frac{1-t^3q^{\la_1+\la_2-2\ta}}{1-t^3q^{\la_1+\la_2-\ta}}.
\end{align*}
We thus recover exactly Perelomov, Ragoucy and Zaugg's result given in
\cite[Theorem~1(a)]{PRZ}.

\subsection{The root system $A_3$}

{}For $k=1,2,3$ our formulas in Theorem~\ref{theorec} write respectively as
\begin{align*}
P_{\la_1\om_1+\la_2\om_2+\la_3\om_3}&=
\sum_{(i,j)\in \mathbb{N}^2}
C^{(1)}_{ij}(\la)\,P_{(\la_1-i-j)\om_1}\,
P_{(\la_2-i)\om_2+(\la_3+i-j)\om_3},\\
&=\sum_{(i,j)\in \mathbb{N}^2}
C^{(2)}_{ij}(\la)\,P_{(\la_2-i-j)\om_1}\,
P_{(\la_1+\la_2+i)\om_1+(\la_3-j)\om_3},\\
&=\sum_{(i,j)\in \mathbb{N}^2}
C^{(3)}_{ij}(\la)\,P_{(\la_3-i-j)\om_1}\,
P_{(\la_1+i-j)\om_1+(\la_2+\la_3+j)\om_2}.
\end{align*}
In order to make these expansions explicit, we need to evaluate the determinant of the 2 by 2 matrix $A$ given by
\begin{equation*}
A_{kl}=v_k^{2-l}
\left(1-t^{l-1}\frac{1-tv_k}{1-v_k}
\frac{v_k-u_1}{v_k-tu_1}\,\frac{v_k-u_2}{v_k-tu_2}\right),
\end{equation*}
with $v_1=q^iu_1,v_2=q^ju_2$.

More precisely we need to compute the quotient of this determinant by the Vandermonde determinant $v_1-v_2=q^i u_1-q^j u_2$. There is no evidence this quotient may be written in canonical form. Inspired by the explicit result of~\cite[Theorem~1]{La} (see below), we write this quotient of determinants as
\begin{multline*}
\frac{\det\!\ A}{q^i u_1-q^j u_2}=
\frac{(t-1)^2}{(t-q^i)(t-q^j)}\Bigg(\frac{1-q^{2i}u_1}{1-q^iu_1}\,\frac{1-q^{2j}u_2}{1-q^ju_2}
\Big(1+t^{-1}\frac{1-q^i}{1-q^iu_1/tu_2}\,\frac{1-q^j} {1-q^ju_2/tu_1}\Big)\\
 -(q^iu_1+q^ju_2)\frac{1-q^i}{1-q^iu_1}\,
\frac{1-q^j}{1-q^ju_2}\, \frac{1-q^i/t}{1-q^iu_1/tu_2}\,
\frac{1-q^j/t}{1-q^ju_2/tu_1}\Bigg).
\end{multline*}
The above identity (which is not trivial) may be easily verified
by using any formal calculus software.

Next, for $(i,j)\in \mathbb{N}^2$ we define
\begin{multline*}
\nabla_{ij}(u_0,u_1,u_2)=\\
q^{i+j}\,
\frac{(t^2u_0;q)_{i+j}}{(qtu_0;q)_{i+j}}\,
\frac{(1/t;q)_i}{(q;q)_i}\,
\frac{(u_1;q)_i}{(qtu_1;q)_i}\,
\frac{(1/t;q)_j}{(q;q)_j}\,
\frac{(u_2;q)_j}{(qtu_2;q)_j}\,
\frac{(q^{i-j+1}u_1/tu_2;q)_j}
{(q^{i-j+1}u_1/u_2;q)_j}\,
\frac{(tq^{-j}u_1/u_2;q)_j}
{(q^{-j}u_1/u_2;q)_j}\\\times
\Bigg( \frac{1-q^{2i}u_1}{1-u_1}\,\frac{1-q^{2j}u_2}{1-u_2}
\Big(1+t^{-1}\frac{1-q^i}{1-q^iu_1/tu_2}\,\frac{1-q^j} {1-q^ju_2/tu_1}\Big)\\
 -(q^iu_1+q^ju_2)\frac{1-q^i}{1-u_1}\,
\frac{1-q^j}{1-u_2}\, \frac{1-q^i/t}{1-q^iu_1/tu_2}\,
\frac{1-q^j/t}{1-q^ju_2/tu_1}\Bigg).
\end{multline*}
It is readily verified that we have
\[
\frac{C_{ij} (u_0,u_1,u_2;1,r)}{\nabla_{ij}(u_0,u_1,u_2)}=
\frac{(tu_1/u_0;q)_i}{(qu_1/u_0;q)_i}
\frac{(tu_2/u_0;q)_j}{(qu_2/u_0;q)_j},
\]
\[
\frac{C_{ij} (u_0,u_1,u_2;2,r)}{\nabla_{ij}(u_0,u_1,u_2)}=
\frac{(u_1/tu_0;q)_i}{(qu_1/t^2u_0;q)_i}\,
\frac{(qtu_0/u_1;q)_r}{(t^2u_0/u_1;q)_r}\,
\frac{(tu_0/u_1;q)_r}{(qu_0/u_1;q)_r}\,
\frac{(tu_2/u_0;q)_j}{(qu_2/u_0;q)_j},
\]
\begin{multline*}
\frac{C_{ij} (u_0,u_1,u_2;3,r)}{\nabla_{ij}(u_0,u_1,u_2)}=
\frac{(u_1/tu_0;q)_i}{(qu_1/t^2u_0;q)_i}\,
\frac{(qtu_0/u_1;q)_r}{(t^2u_0/u_1;q)_r}\,
\frac{(tu_0/u_1;q)_r}{(qu_0/u_1;q)_r}\\\times
\frac{(u_2/tu_0;q)_j}{(qu_2/t^2u_0;q)_j}\,
\frac{(qtu_0/u_2;q)_r}{(t^2u_0/u_2;q)_r}\,
\frac{(tu_0/u_2;q)_r}{(qu_0/u_2;q)_r}.
\end{multline*}
Now, by Theorem~\ref{theorec} the respective recurrence coefficients
are determined to be
\begin{align*}
C^{(1)}_{ij}(\la)&=
C_{ij} (q^{-\la_1}/t^2,q^{-\la_1-\la_2}/t^3,
q^{-\la_1-\la_2-\la_3}/t^4;1,\la_1),\\
C^{(2)}_{ij}(\la)&=
C_{ij} (q^{-\la_2}/t^2,q^{\la_1},q^{-\la_2-\la_3}/t^3;2,\la_2),\\
C^{(3)}_{ij}(\la)&=
C_{ij} (q^{-\la_3}/t^2,q^{\la_1+\la_2}t,q^{\la_2};3,\la_3).
\end{align*}

The cases $k=1,2$ are new. For $k=3$ we recover the first author's
earlier result in \cite[Theorem~1]{La},
more precisely the restriction of this result to four variables
$(x_1,x_2,x_3,x_4)$ subject to $x_1x_2x_3x_4=1$. 
Namely given a partition $(\mu_1,\mu_2,\mu_3)$ and $u=q^{\mu_1-\mu_2}$, $v=q^{\mu_2-\mu_3}$, the Macdonald
symmetric function $\mathcal{P}_{(\mu_1,\mu_2,\mu_3)}(q,t)$ is given by
\begin{equation*}
\mathcal{P}_{(\mu_1,\mu_2,\mu_3)}= \sum_{(i,j)\in \mathbb{N}^2}  
\mathcal{C}_{ij} (\mu)  \, \mathcal{P}_{(\mu_3-i-j)}\,\mathcal{P}_{(\mu_1+i,\mu_2+j)},
\end{equation*}
with
\begin{align*}
\mathcal{C}_{ij} (\mu) &=t^{i+j}\, 
\frac{(1/t;q)_i}{(q;q)_i}\,\frac{(1/t;q)_j}{(q;q)_j}\, \frac{(tuv;q)_i}{(qt^2uv;q)_i}\,\frac{(v;q)_j}{(qtv;q)_j}\, 
\frac{{(q^{-j}t^2u;q)}_{i}}{{(q^{-j}tu;q)}_{i}} \,
\frac{{(qu;q)}_{i}}{{(qtu;q)}_{i}}\\ &\times
\frac{(t;q)_{\mu_1-\mu_{2}+i-j}}
{(q;q)_{\mu_1-\mu_{2}+i-j}}\,
\frac{(t;q)_{\mu_2+j}}
{(q;q)_{\mu_2+j}}\,
\frac{(t;q)_{\mu_3-i-j}}
{(q;q)_{\mu_3-i-j}}\,\frac{(q;q)_{\mu_1-\mu_{2}}}{(t;q)_{\mu_1-\mu_{2}}}\,
\frac{(q;q)_{\mu_2-\mu_{3}}}{(t;q)_{\mu_2-\mu_{3}}}\,
\frac{(q;q)_{\mu_3}}{(t;q)_{\mu_3}}
\\ &\times
\frac{(q^{i-j}t^{2}u;q)_{\mu_2+j}}
{(q^{i-j+1}tu;q)_{\mu_2+j}}\,
\frac{(qtu;q)_{\mu_2-\mu_{3}}}
{(t^{2}u;q)_{\mu_2-\mu_{3}}}\,
\frac{(qt^{2}uv;q)_{\mu_3}}{(t^{3}uv;q)_{\mu_3}}\,
\frac{(qtv;q)_{\mu_3}}{(t^{2}v;q)_{\mu_3}}\,
\frac{1-q^{2i}tuv}{1-tuv}\,\frac{1-q^{2j}v}{1-v}
\\ &\times
\Bigg(1 +u \, \frac{1-q^i}{1-q^{i}u} \, \frac{1-q^{-j}}{1-q^{-j}t^2u}  
\Big(t- v(q^itu+q^j) \frac{t-q^i}{1-q^{2i}tuv} \frac{t-q^j}{1-q^{2j}v} \Big)\Bigg).
\end{align*}
The reader may check our formula is indeed equivalent to
\cite[Theorem~1]{La} by using the correspondence
$\la_1=\mu_1-\mu_2$, $\la_2=\mu_2-\mu_3$, $\la_3=\mu_3$
between dominant weights and partitions.

\section{Final remark}

The Macdonald polynomial $P_\la$, $\la= \sum_{i=1}^n \la_i\om_i$, is in bijective correspondence with the symmetric function $\mathcal{P}_\mu(x_1,\ldots,x_{n+1})$ with $\mu=(\mu_1,\ldots, \mu_n)$, $\mu_i=\sum_{j=i}^n \la_j$, subject to the condition $x_1\cdots x_{n+1}=1$. Therefore the $n$ recurrence relations that we have obtained for $P_\la$ may be expressed in terms of $\mathcal{P}_\mu(x_1,\ldots,x_{n+1})$, subject to $x_1\cdots x_{n+1}=1$.

One may wonder whether this restriction can be removed. Equivalently, being given some fixed integer $1\le k \le n$, is it possible to expand the symmetric function $\mathcal{P}_\mu$ in terms of products $\mathcal{P}_{(r)}\mathcal{P}_\rho$ for partitions $\rho=(\rho_1,\ldots, \rho_n)$ satisfying $\rho_k=\rho_{k+1}$?

Such a development has been obtained in~\cite{LS} for $k=n$, in which case $\rho_n=\rho_{n+1}=0$. However this method cannot be used for other values of $k$.

Actually the Pieri expansion of $\mathcal{P}_{(r)}\mathcal{P}_\rho$ involves symmetric functions $\mathcal{P}_\sigma$ with $\sigma-\rho$ a horizontal $r$-strip. Hence some of these partitions $\sigma$ have length $l(\sigma)= n+1$. The only exception occurs for $k=n$ since in that case $\rho_n=0$ entails $l(\sigma) \le n$.

Therefore, except for $k=n$, the Pieri multiplication does not conserve the space generated by $\{\mathcal{P}_\kappa, l(\kappa) \le n\}$, and it is not possible to define a Pieri matrix to invert.

This difficulty does not arise in the $A_n$ framework. Then the Pieri matrix can be defined, because the condition $x_1\cdots x_{n+1}=1$ and the property~\cite[(4.17), p. 325]{Ma2}
\[\mathcal{P}_{(\sigma_1,\ldots,\sigma_{n+1})}(x_1,\ldots,x_{n+1})=
(x_1\cdots x_{n+1})^{\sigma_{n+1}}\,\mathcal{P}_{(\sigma_1-\sigma_{n+1},\ldots,\sigma_{n}-\sigma_{n+1},0)}(x_1,\ldots,x_{n+1})\]
allow to deal with partitions of length $n+1$. 

\section*{Appendix: A multidimensional matrix inverse}

The following result (equivalent to one previously given
in \cite{LS}) is crucial to obtain the recursion
formula in Section~\ref{secrec}.

\begin{Lemma}
Let $t,u_0,u_1,\dots,u_n$ be indeterminates and
$r,k\in \mathbb{N}$ with $1\le k\le n+1$. Define
\begin{multline*}
f_{\ba\ka}=
q^{|\ba|-|\ka|}\,
\frac{(t^2u_0;q)_{|\ba|}}{(qtu_0;q)_{|\ba|}}\,
\frac{(qtu_0;q)_{|\ka|}}{(t^2u_0;q)_{|\ka|}}\,
\prod_{i=1}^n\frac{(q/t;q)_{\ba_i-\ka_i}}{(q;q)_{\ba_i-\ka_i}}\,
\frac{(q^{\ka_i+|\ka|+1}u_i;q)_{\ba_i-\ka_i}}
{(q^{\ka_i+|\ka|+1}tu_i;q)_{\ba_i-\ka_i}}\\\times
\prod_{i=1}^{k-1}\frac{(u_i/tu_0;q)_{\ba_i}}{(qu_i/t^2u_0;q)_{\ba_i}}\,
\frac{(qu_i/tu_0;q)_{\ka_i}}{(u_i/u_0;q)_{\ka_i}}\,
\frac{(u_i/u_0;q)_{|\ka|-r+\ka_i}}{(qu_i/tu_0;q)_{|\ka|-r+\ka_i}}\,
\frac{(qu_i/t^2u_0;q)_{|\ka|-r+\ka_i}}{(u_i/tu_0;q)_{|\ka|-r+\ka_i}}
\\ \hspace{4 cm} \times\prod_{i=k}^n
\frac{(tu_i/u_0;q)_{\ba_i}}{(qu_i/u_0;q)_{\ba_i}}\,
\frac{(qu_i/u_0;q)_{\ka_i}}{(tu_i/u_0;q)_{\ka_i}}
\\\times
\prod_{1\le i<j\le n}\frac{(q^{\ba_i-\ba_j+1}u_i/tu_j;q)_{\ba_j-\ka_j}}
{(q^{\ba_i-\ba_j+1}u_i/u_j;q)_{\ba_j-\ka_j}}\,
\frac{(q^{\ka_i-\ba_j}tu_i/u_j;q)_{\ba_j-\ka_j}}
{(q^{\ka_i-\ba_j}u_i/u_j;q)_{\ba_j-\ka_j}}
\: \big(q^{\ba_i}u_i-q^{\ba_j}u_j\big)^{-1}
\\\times
\det_{1\le i,j\le n}\!\Bigg[\big(q^{\ba_i}u_i\big)^{n-j}\Bigg(1-
t^{j-1}
\frac{\big(1-q^{\ba_i+|\ka|}tu_i\big)}
{\big(1-q^{\ba_i+|\ka|}u_i\big)}
\prod_{s=1}^n\frac{\big(q^{\ba_i}u_i-q^{\ka_s}u_s\big)}
{\big(q^{\ba_i}u_i-q^{\ka_s}tu_s\big)}\Bigg)\Bigg],
\end{multline*}
and
\begin{multline*}
g_{\ka\ga}=
\Big(\frac qt\Big)^{|\ka|-|\ga|}\,
\frac{(t^2u_0;q)_{|\ka|}}{(qtu_0;q)_{|\ka|}}\,
\frac{(qtu_0;q)_{|\ga|}}{(t^2u_0;q)_{|\ga|}}\,
\prod_{i=1}^n\frac{(t;q)_{\ka_i-\ga_i}}{(q;q)_{\ka_i-\ga_i}}\,
\frac{(q^{\ga_i+|\ka|+1}u_i;q)_{\ka_i-\ga_i}}
{(q^{\ga_i+|\ka|}tu_i;q)_{\ka_i-\ga_i}}
\\\times
\prod_{i=1}^{k-1}
\frac{(u_i/u_0;q)_{\ka_i}}{(qu_i/tu_0;q)_{\ka_i}}\,
\frac{(qu_i/t^2u_0;q)_{\ga_i}}{(u_i/tu_0;q)_{\ga_i}}\,
\frac{(qu_i/tu_0;q)_{|\ka|-r+\ka_i}}{(u_i/u_0;q)_{|\ka|-r+\ka_i}}\,
\frac{(u_i/tu_0;q)_{|\ka|-r+\ka_i}}{(qu_i/t^2u_0;q)_{|\ka|-r+\ka_i}}\\
\hspace{4 cm} \times
\prod_{i=k}^n
\frac{(tu_i/u_0;q)_{\ka_i}}{(qu_i/u_0;q)_{\ka_i}}\,
\frac{(qu_i/u_0;q)_{\ga_i}}{(tu_i/u_0;q)_{\ga_i}}\\\times
\prod_{1\le i<j\le n}\frac{(q^{\ka_i-\ka_j}tu_i/u_j;q)_{\ka_j-\ga_j}}
{(q^{\ka_i-\ka_j+1}u_i/u_j;q)_{\ka_j-\ga_j}}\,
\frac{(q^{\ga_i-\ka_j+1}u_i/tu_j;q)_{\ka_j-\ga_j}}
{(q^{\ga_i-\ka_j}u_i/u_j;q)_{\ka_j-\ga_j}}.
\end{multline*}
Then the infinite lower-triangular $n$-dimensional matrices
$(f_{\ba\ka})_{\ba,\ka\in\mathbf Z^n}$ and
$(g_{\ka\ga})_{\ka,\ga\in\mathbf Z^n}$
are mutually inverse.
\end{Lemma}

\begin{proof}
Given two non-zero sequences $(\xi_{\ka})$ and $(\zeta_{\ka})$,
and a pair of matrices $(f_{\ba\ka})$ and $(g_{\ka\ga})$
which are mutually inverse, it is easily checked
(using the trivial relation
$\frac{\xi_{\ba}}{\xi_{\ga}}\delta_{\ba\ga}=\delta_{\ba\ga}$)
that the matrices
$(f_{\ba\ka}\, \xi_{\ba}/\zeta_{\ka})$ and
$(g_{\ka\ga}\, \zeta_{\ka}/\xi_{\ga})$ are mutually inverse.

We choose
\begin{multline*}
\xi_{\ka}=
\Big(\frac qt\Big)^{|\ka|}\,
\frac{(t^2u_0;q)_{|\ka|}}{(qtu_0;q)_{|\ka|}}\,
\prod_{i=1}^{k-1}
\frac{(u_i/tu_0;q)_{\ka_i}}{(qu_i/t^2u_0;q)_{\ka_i}}\,
\prod_{i=k}^n
\frac{(tu_i/u_0;q)_{\ka_i}}{(qu_i/u_0;q)_{\ka_i}}\\\times
\prod_{1\le i<j\le n}\frac{(qu_i/u_j;q)_{\ka_i-\ka_j}}
{(tu_i/u_j;q)_{\ka_i-\ka_j}}\,
\frac{(u_i/u_j;q)_{\ka_i-\ka_j}}
{(qu_i/tu_j;q)_{\ka_i-\ka_j}},
\end{multline*}
\begin{multline*}
\zeta_{\ka}=
\Big(\frac qt\Big)^{|\ka|}\,
\frac{(t^2u_0;q)_{|\ka|}}{(qtu_0;q)_{|\ka|}}\,
\prod_{i=1}^{k-1}
\frac{(u_i/u_0;q)_{\ka_i}}{(qu_i/tu_0;q)_{\ka_i}}\,
\frac{(qu_i/tu_0;q)_{|\ka|-r+\ka_i}}{(u_i/u_0;q)_{|\ka|-r+\ka_i}}\,
\frac{(u_i/tu_0;q)_{|\ka|-r+\ka_i}}{(qu_i/t^2u_0;q)_{|\ka|-r+\ka_i}}\\\times
\prod_{i=k}^n
\frac{(tu_i/u_0;q)_{\ka_i}}{(qu_i/u_0;q)_{\ka_i}}\,
\prod_{1\le i<j\le n}\frac{(qu_i/u_j;q)_{\ka_i-\ka_j}}
{(tu_i/u_j;q)_{\ka_i-\ka_j}}\,
\frac{(u_i/u_j;q)_{\ka_i-\ka_j}}
{(qu_i/tu_j;q)_{\ka_i-\ka_j}},
\end{multline*}
together with the pair of mutually inverse matrices
$(f_{\ba\ka})$ and $(g_{\ka\ga})$ as defined in
\cite[Theorem~2.7]{LS}.

Several elementary manipulations of $q$-shifted factorials eventually
lead to the result in the desired form. To give a sample,
(concentrating only on the products over
$\prod_{1\le i<j\le n}$ of $q$-shifted factorials) we use the
simplification
\begin{align*}
\prod_{1\le i<j\le n}\frac{(q^{\ka_i-\ka_j+1}u_i/tu_j;q)_{\ba_i-\ka_i}}
{(q^{\ka_i-\ka_j+1}u_i/u_j;q)_{\ba_i-\ka_i}}
\frac{(q^{\ka_i-\ba_j}tu_i/u_j;q)_{\ba_i-\ka_i}}
{(q^{\ka_i-\ba_j}u_i/u_j;q)_{\ba_i-\ka_i}}\quad&\\\times
\prod_{1\le i<j\le n}\frac{(qu_i/u_j;q)_{\ba_i-\ba_j}}
{(tu_i/u_j;q)_{\ba_i-\ba_j}}\frac{(u_i/u_j;q)_{\ba_i-\ba_j}}
{(qu_i/tu_j;q)_{\ba_i-\ba_j}}\frac{(tu_i/u_j;q)_{\ka_i-\ka_j}}
{(qu_i/u_j;q)_{\ka_i-\ka_j}}\frac{(qu_i/tu_j;q)_{\ka_i-\ka_j}}
{(u_i/u_j;q)_{\ka_i-\ka_j}}=&\\
\prod_{1\le i<j\le n}\frac{(qu_i/tu_j;q)_{\ba_i-\ka_j}}
{(qu_i/u_j;q)_{\ba_i-\ka_j}}\frac{(u_i/u_j;q)_{\ka_i-\ba_j}}
{(tu_i/u_j;q)_{\ka_i-\ba_j}}\frac{(qu_i/u_j;q)_{\ba_i-\ba_j}}
{(qu_i/tu_j;q)_{\ba_i-\ba_j}}\frac{(tu_i/u_j;q)_{\ka_i-\ka_j}}
{(u_i/u_j;q)_{\ka_i-\ka_j}}=&\\
\prod_{1\le i<j\le n}\frac{(q^{\ba_i-\ba_j+1}u_i/tu_j;q)_{\ba_j-\ka_j}}
{(q^{\ba_i-\ba_j+1}u_i/u_j;q)_{\ba_j-\ka_j}}
\frac{(q^{\ka_i-\ba_j}tu_i/u_j;q)_{\ba_j-\ka_j}}
{(q^{\ka_i-\ba_j}u_i/u_j;q)_{\ba_j-\ka_j}}&
\end{align*}
in the computation of $f_{\ba\ka}$ in the Lemma.
\end{proof}

\end{document}